\documentclass[12pt]{amsart}
\usepackage{amsmath}
\usepackage{amssymb}
\usepackage{amscd}
\usepackage[only,llbracket,rrbracket]{stmaryrd}
\newtheorem{theorem}{Theorem}[section]
\newtheorem{lemma}[theorem]{Lemma}

\theoremstyle{definition}
\newtheorem{definition}[theorem]{Definition}
\newtheorem{definitions}[theorem]{Definitions}

\newtheorem{definitions and remarks}[theorem]{Definitions and Remarks}

\theoremstyle{remark}
\newtheorem{remark}[theorem]{Remark}
\newtheorem{remarks}[theorem]{Remarks}

\numberwithin{equation}{section}

\newcommand{\inv}{\mathrm{inv}}

\newcommand{\Sing}{\mathrm{Sing}\,}

\newcommand{\ord}{\mathrm{ord}}

\newcommand{\car}{\mathrm{char}\,}

\newcommand{\Gal}{\text{Gal}}

\newcommand{\al}{{\alpha}}
\newcommand{\be}{{\beta}}
\newcommand{\de}{{\delta}}
\newcommand{\ep}{{\epsilon}}
\newcommand{\D}{{\Delta}}
\newcommand{\g}{{\gamma}}

\newcommand{\s}{{\sigma}}

\newcommand{\IN}{{\mathbb N}}

\newcommand{\IA}{{\mathbb A}}

\newcommand{\IC}{{\mathbb C}}

\newcommand{\IZ}{{\mathbb Z}}

\newcommand{\cO}{{\mathcal O}}
\newcommand{\cQ}{{\mathcal Q}}

\newcommand{\cS}{{\mathcal S}}

\newcommand{\uk}{\underline{k}}

\newcommand{\tI}{{\tilde I}}

\newcommand{\wcO}{{\widehat \cO}}

\newcommand{\llbr}{{\llbracket}}
\newcommand{\rrbr}{{\rrbracket}}
\newcommand{\llb}{{(\!(}}
\newcommand{\rrb}{{)\!)}}

\begin{document}

\title[Resolution except for minimal singularities II]
{Resolution except for minimal singularities\\II. The case of four variables}

\author{Edward Bierstone}
\address{The Fields Institute, 222 College Street, Toronto, Ontario, Canada
M5T 3J1, and University of Toronto, Department of Mathematics, 40 St. George Street,
Toronto, Ontario, Canada M5S 2E4}
\email{bierston@fields.utoronto.ca}
\thanks{The authors' research was supported in part by the following grants:
Bierstone: NSERC OGP0009070 and MRS342058, Lairez: NSERC OGP0009070 
and MRS342058, Milman: NSERC OGP0008949.}

\author{Pierre Lairez}
\address{Ecole Normale Sup\'erieure, D\'epartement de Math\'ematiques et
Applications, 45 rue d'Ulm, 75005 Paris, France}
\email{pierre.lairez@normalesup.org}

\author{Pierre D. Milman}
\address{Department of Mathematics, University of Toronto, 40 St. George Street,
Toronto, Ontario, Canada M5S 2E4}
\email{milman@math.toronto.edu}



\begin{abstract}
In this sequel to \cite{BMminI}, we find the smallest class of singularities
in four variables with which we necessarily end up if we resolve singularities except
for normal crossings. The main new feature is a
characterization of singularities in four variables which occur as limits of triple normal
crossings singularities, and which cannot be eliminated by a birational morphism
that avoids blowing up normal crossings singularities. This
result develops the philsophy of \cite{BMminI}, that the desingularization invariant 
together with natural geometric information can be used to compute
local normal forms of singularities.
\end{abstract}

\maketitle
\setcounter{tocdepth}{1}
\tableofcontents

\section{Introduction}\label{sec:introII} 
This article is a sequel to \cite{BMminI}. We find the smallest class of singularities
in four variables with which we necessarily end up if we resolve singularities except
for normal crossings. The main feature beyond the techniques of \cite{BMminI} is a
characterization of singularities in four variables which occur as limits of triple normal
crossings singularities, and which cannot be eliminated by a birational morphism
that avoids blowing up the triple normal crossings singularities (Theorem \ref{thm:char}). The
latter result develops the philsophy of \cite{BMminI}, that the desingularization invariant 
of \cite{BMinv} together with natural geometric information can be used to compute
local normal forms of singularities.

The reader is referred to \cite{BMminI} for the background and techniques of 
this article. Throughout the paper, an \emph{algebraic variety} means a scheme of finite 
type over a field $\uk$, and $\car \uk = 0$.

\begin{definitions} \label{def:ncII}
We say that $X$ has \emph{normal
crossings} at a point $a$ if, locally at $a$, $X$ can be embedded in a smooth
variety $Z$ with local \'etale coordinates $(x_1,\ldots, x_n)$ at $a$ in which $X$ is defined by
a monomial equation 
\begin{equation}\label{eq:monII}
x_1^{\al_1}\cdots x_n^{\al_n} = 0
\end{equation}
(where the $\al_i$ are nonnegative integers). 
We will say that $X$ has \emph{normal crossings of order} $k$ (or nc$k$) at $a$ if precisely 
$k$ exponents $\al_i$ are nonzero in (\ref{eq:monII}). A singularity $xy = 0$ is called
\emph{double normal crossings} (nc2), a singularity $xyz = 0$ \emph{triple normal
crossings} (nc3), etc. Let $X^{\mathrm{nc}}$ denote the locus of points of $X$ having only normal
crossings singularities. ($X^{\mathrm{nc}}$ includes all smooth points of $X$.)
\end{definitions}

A variety $X$ has normal crossings at $a$ if and only if it can be defined at $a$ by a monomial
equation with respect to formal coordinates, after a finite extension of the ground field $\uk$.

\begin{definition} \label{def:SII}
Let $\cS$ denote the following class of singularities in four variables $(w,x,y,z)$:
\medskip

\renewcommand{\arraystretch}{1.1}
\begin{tabular}{r l}
                   $xy = 0$ & \quad nc2 \\
                 $xyz = 0$ & \quad nc3 \\
               $xyzw = 0$ & \quad nc4 \\
      $z^2 + xy^2 = 0$ & \quad \emph{pinch point} pp \\ 
      $z^2 + (y + 2x^2)(y - x^2)^2 = 0$ & \quad \emph{degenerate pinch point} dpp \\
      $x(z^2 + wy^2) = 0$ & \quad \emph{product} prod \\
      $z^3 + wy^3 + w^2x^3 -3wxyz = 0$ & \quad \emph{cyclic point} cp3 \\
\end{tabular}
\end{definition}

In other words, $\cS$ is the class of singularities that can be written in local \'etale
coordinates (or in formal variables after finite field extension) as one of the 
\emph{normal forms} in the preceding table. 

Let $X^\cS$ denote
the locus of points of $X$ having only singularities in $\cS$. In other words, if 
$X$ is an algebraic variety of dimension three, then $X^\cS$ is the locus
of points $a$ of $X$ such that either $a$ is a smooth point or $a$ has a neighbourhood 
$U$ where $X|_U$ admits an embedding $X|_U \hookrightarrow Z$ in a smooth 
4-dimensional variety $Z$, and $X$ has a singularity in $\cS$ at $a$, with respect to suitable
local \'etale coordinates of $Z$.

All singularities in $\cS$ are hypersurface singularities. We say that
$X$ is a \emph{hypersurface} if, locally, $X$ can be defined by a principal ideal on
a smooth variety. (We say that $X$ is an \emph{embedded hypersurface} if $X
\hookrightarrow Z$, where $Z$ is smooth and $X$ is defined by a principal ideal on $Z$.)

\begin{theorem}\label{thm:mincompl4var}
Let $X$ denote a reduced variety of pure dimension $3$. Then
there is a morphism $\s: X' \to X$ given by a finite sequence of admissible blowings-up
\begin{equation}\label{eq:blupX.4}
X = X_0 \stackrel{\s_1}{\longleftarrow} X_1 \longleftarrow \cdots
\stackrel{\s_{t}}{\longleftarrow} X_{t} = X'\,,
\end{equation}
such that
\begin{enumerate}
\item[(a)] $X' = (X')^{\cS}$,
\item[(b)] $\s$ is an isomorphism over $X^{\mathrm{nc}}$.
\end{enumerate}
Moreover, the morphism $\s = \s_X$ (or the entire blowing-up sequence
\eqref{eq:blupX.4}) can be realized in a way that is functorial with respect
to \'etale morphisms.
\end{theorem}

See \cite[Rem.\,1.15]{BMminI} on minimality of the class $\cS$.

An \emph{admissible} blowing-up means a blowing-up $\s$ with centre $C$ which is
smooth and has only simple normal crossings with respect to the exceptional
divisor. The latter condition means that, with respect to a suitable local embedding
of $X$ in a smooth variety $Z$ and the induced blowing-up sequence of $Z$,
there are regular coordinates (i.e., regular parameters) 
$(x_1,\ldots, x_n)$ at any point of $C$, in which 
$C$ is a coordinate subspace and each component of the exceptional divisor is a
coordinate hyperplane $(x_i = 0)$, for some $i$.

\begin{remark}\label{rem:hypersurf}
Theorem \ref{thm:mincompl4var} can be reduced to the case of a hypersurface using the strong desingularization algorithm of \cite{BMinv, BMfunct}. The algorithm involves 
blowing up with smooth centres in the maximum strata of the Hilbert-Samuel function. 
The latter determines the local embedding dimension, so the algorithm first eliminates 
points of embedding codimension $> 1$ without modifying normal crossings points
(or points with singularities in $\cS$).
\end{remark}

\begin{theorem}\label{thm:minstrongcompl4var}
Let $\cS'$ denote the class of singularities $\cS$ together with the following
singularity:
\begin{equation}\label{eq:excII}
z^2 + y(wy + x^2)^2 = 0 \quad \text{\emph{exceptional singularity} \emph{exc}}
\end{equation}
Let $X$ denote a reduced variety of pure 
dimension $3$. Then
there is a morphism $\s: X' \to X$ given by a finite sequence of admissible 
blowings-up \eqref{eq:blupX.4}
such that
\begin{enumerate}
\item[(a)] $X' = (X')^{\cS'}$,
\item[(b)] $\s$ is an isomorphism over $X^{\cS'}$.
\end{enumerate}
Moreover, the morphism $\s = \s_X$ (or the entire blowing-up sequence
\eqref{eq:blupX.4}) can be realized in a way that is functorial with respect
to \'etale morphisms.
\end{theorem}

\begin{remark}\label{rem:exc} \emph{Exceptional singularity.}
The exceptional singularity is a limit of dpp singularities that cannot be eliminated 
without blowing up the dpp-locus: 

The equation \eqref{eq:excII} defines an embedded hypersurface $X \hookrightarrow
Z := \IA^4_{(w,x,y,z)}$. Outside the origin, $X$ has only smooth points, 2nc singularities
(when $z= wy+x^2 = 0$, $y\neq 0$), and degenerate pinch points dpp (when $z = y
= wy+x^2 = 0$, $w \neq 0$). Any birational morphism $Z' \to Z$ (where $Z'$ is smooth)
which eliminates the exceptional singularity at $0$ and 
is an isomorphism over the complement of $0$, factors through the blowing-up of $0$.

Let $\s: Z' \to Z$ denote the blowing-up of $0$ and let $X'$ be the strict transform of $X$
by $\s$. In the chart of $Z'$ with coordinates $(w,x,y,z)$ in which $\s$ is given by
$(w,wx,wy,wz)$, $X'$ is defined by the equation $z^2 + w^3y(y+x^2)^2 = 0$. After a
``cleaning'' blowing-up (centre $(z=w=0)$; see \cite[Sect.\,2]{BMminI}), we get
\begin{equation}\label{eq:B}
z^2 + wy(y+x^2)^2 = 0.
\end{equation}

The hypersurface \eqref{eq:B} has quadratic cone singularities when $z=y=w=0$, $x\neq 0$,
pinch points pp when $z = w = y+x^2 = 0$, $y \neq 0$, and degenerate pinch points
dpp when $z=y= y+x^2 = 0$, $w\neq 0$.
Any birational morphism that preserves singularities in $\cS$ factors through the 
blowing-up either of $(z=y=w=0)$ or of $0$. The former leads to another exceptional
singularity, while the latter leads to another singularity of type \eqref{eq:B} after cleaning.
\end{remark}

Theorems \ref{thm:mincompl4var} and \ref{thm:minstrongcompl4var} will be proved
in Section 4. Our proofs also give normal forms or local models for the singularites of the
total transform of $X$, corresponding to $\cS$ or $\cS'$. 
(Equivalently, they give local models for the
\emph{transform} of a divisor $D \subset Z$ ($\dim Z = 4$), where the transform is defined 
as the support of the birational transform
plus the exceptional divisor). See \cite[Sect.\,1]{BMminI} and Section \ref{sec:dim4} below.
\smallskip

The results in this article form part of the subject of Pierre Lairez's 
M\'emoire de Magist\`ere
at the Ecole Normale Sup\'erieure. The authors are grateful to Franklin Vera Pacheco
for many important comments.

\subsection{Limits of triple normal crossings points}\label{subsec:nc3} Our proofs
of Theorems \ref{thm:mincompl4var} and \ref{thm:minstrongcompl4var} 
are based on using the desingularization invariant 
of \cite{BMinv} as a tool for computing and simplifying local normal forms.  
In \cite[Sects.\,1,5]{BMminI}, we try to provide a working knowledge of the desingularization
algorithm and the invariant as they are used here, for a reader not necessarily familiar
with a complete proof of resolution of singularities. The reader is referred to the latter
for more details of the notions below. 

Suppose that $X \hookrightarrow Z$
is an embedded hypersurface, where $Z$ is smooth. Let $\inv = \inv_X$ denote the
desingularization invariant for $X$. We recall that $\inv$ is defined iteratively on the
strict transform $X_{j+1}$ of $X =X_0$ for any finite sequence of $\inv$-\emph{admissible} 
blowings-up 
\begin{equation}\label{eq:finblupII}
Z = Z_0 \stackrel{\s_1}{\longleftarrow} Z_1 \longleftarrow \cdots
\stackrel{\s_{j+1}}{\longleftarrow} Z_{j+1}\,.
\end{equation}
(A blowing-up is $\inv$-\emph{admissible} if it is admissible and $\inv$ is constant
on each component of its centre.) In particular, $\inv(a)$, where $a \in X_{j+1}$ 
depends not only on $X_{j+1}$
but also on the \emph{history} of blowings-up (\ref{eq:finblupII}).

Let $a \in X_j$. then $\inv(a)$ has the form
\begin{equation}\label{eq:invintroII}
\inv(a) = (\nu_1(a), s_1(a), \ldots, \nu_t(a), s_t(a), \nu_{t+1}(a))\,,
\end{equation}
where $\nu_k(a)$ is a positive rational number (called a \emph{residual multiplicity}) if $k\leq t$, 
each $s_k(a)$ is a nonnegative integer (which counts certain components of the exceptional
divisor), and $\nu_{t+1}(a)$ is either $0$ or $\infty$.
The successive pairs $(\nu_k(a),s_k(a))$ are defined inductively over
\emph{maximal contact} subvarieties of increasing codimension.

It is easy to see that, in \emph{year zero} (i.e., if $j=0$), then $\inv(a) = (2,0,1,0,\infty)$
if and only if $X$ has a double normal crossings singularity $z^2 + y^2 = 0$ at $a$.
Following are some other year-zero hypersurface examples:
\medskip

\renewcommand{\arraystretch}{1.1}
\begin{tabular}{r c l}
                   $x = 0$ & \,smooth\, & $\inv(0) = \inv(\mathrm{nc}1) := (1,0,\infty)$\\
               $x_1x_2\cdots x_k = 0$ & nc$k$ & 
                                  $\inv(0) = \inv(\mathrm{nc}k) := (k,0,1,0,\ldots,1,0,\infty)$\\
   $z^2 + xy^2 = 0$ & pp & $\inv(0) = \inv(\mathrm{pp}) := (2,0,3/2,0,1,0,\infty)$
\end{tabular}
\medskip

\noindent
(where, for nc$k$, there are $k-1$ pairs $(1,0)$). For $k \geq 3$, nc$k$ is 
not characterized by the value of $\inv$; for example the singularity 
$x_1^k + x_2^k + \cdots + x_k^k = 0$ also has $\inv(0) = (k,0,1,0,\ldots,1,0,\infty)$
with  $k-1$ pairs $(1,0)$.

Let $X \hookrightarrow Z$ denote an embedded hypersurface, where $\dim Z = 4$.
As above, if $a$ is a triple normal crossings point of $X$, then 
$\inv(a) = \inv(\text{nc}3) := (3,0,1,0,1,0,\infty)$ (this is ``year zero'').
Consider the desingularization algorithm applied to $X \subset Z$.
In Theorem \ref{thm:char} following, 
we provide normal forms for the singularities which can occur at special points of 
a component $C$ of the locus
\begin{equation}\label{eq:invnc3}
(\inv = \inv(\text{nc}3))
\end{equation}
in the strict transform $X_{j_0}$, for any year $j_0$ of the resolution history 
\eqref{eq:finblupII},
assuming that the generic point of $C$ is nc$3$. Theorem \ref{thm:cleaned} below provides normal forms
(in $\cS$) for the singularities we get by applying  \emph{cleaning} blowings-up to simplify
the preceding singularities. (See \cite[Sects.\,1,2]{BMminI}.)

The locus $(\inv = \inv(\text{nc}3)) \subset X_{j_0}$ is a smooth curve. Let $a \in C$, where 
$C$ is a component of $(\inv = \inv(\text{nc}3))$ which is generically nc$3$. Using the
Weierstrass preparation theorem, in suitable
\'etale coordinates $(w,x,y,z)$ at $a = 0$, we can write the equation of $X_{j_0}$ in $Z_{j_0}$
as $f(w,x,y,z) = 0$, where $f$ is nc$3$ on $(x = y = z =0,\, w \neq 0)$, and
\begin{equation}\label{eq:cubic}
f(w,x,y,z) = z^3 + A(w,x,y)z^2 + B(w,x,y)z + C(w,x,y),.
\end{equation}
and we can assume that $A=0$, by ``completing the cube''.

\begin{theorem}\label{thm:char}
Let $X \hookrightarrow Z$ denote an embedded hypersurface, where $Z$ is a smooth 
algebraic variety of pure dimension four. Assume that the ground field $\uk$ is algebraically closed. Consider a finite sequence of $\inv$-admissible
blowings-up \eqref{eq:finblupII}. Let $a \in X_{j_0}$, for some $j=j_0$,
and let $f=0$ be a local defining
equation for $X_{j_0}$ at $a$. Suppose that $Z_{j_0}$ has a regular coordinate system 
$(w,x,y,z)$ at $a=0$ such that:
\begin{enumerate}
\item[(i)] $w=0$ is a local equation for the exceptional divisor (if the latter contains $a$);
\item[(ii)] $X_{j_0}$ is $\mathrm{nc}3$ at every nonzero point of the $w$-axis $(z=y=x=0)$;
\item[(iii)] $\inv(a) = \inv(\mathrm{nc}3) := (3,0,1,0,1,0,\infty)$.
\end{enumerate}
Then:
\begin{enumerate}
\item $f$ has three analytic branches at $a=0$ (i.e., three factors of order $1$ in a 
suitable \'etale neighbourhood) if and only if
\begin{equation*}\label{eq:3br}
f(w,x,y,z) = z\left(z+w^\al x\right)\left(z + w^\al\left(x\xi + w^\be y\right)\right),
\end{equation*}
where $\xi = \xi(w,x,y)$, after a suitable \'etale coordinate change preserving $(w=0)$.
\end{enumerate}
On the other hand, suppose that $f$ does not have three analytic branches at $a$. Then:
\begin{enumerate}
\item[(2)] The following are equivalent:
\smallskip
\begin{enumerate}
\item $f$ has two analytic branches at $a=0$;
\item $f(w^2,x,y,z)$ has three analytic branches at $0$;
\item after an \'etale coordinate change preserving $(w=0)$, we can write $f(w,x,y,z)$ 
either as
\begin{equation*}
\left(z+w^\al x\right)\left(z^2 + w^{2\al +1}\left(x\xi + w^\be y\right)^2\right),
\end{equation*}
where $\xi = \xi(w,x,y)$, or as
\begin{equation*}
\left(z+w^\al\left(y\eta + w^\be x\right)\right)\left(z^2 + w^{2\al +1}y^2\right),
\end{equation*}
where $\eta = \eta(w,x,y)$.
\smallskip
\end{enumerate}
\item[(3)] The following are equivalent:
\smallskip
\begin{enumerate}
\item $f$ is analytically irreducible at $a=0$;
\item $f(w^3,x,y,z)$ has three analytic branches at $0$;
\item after an \'etale coordinate change preserving $(w=0)$, we can write $f(w,x,y,z)$ as
\begin{equation*}
z^3 - 3w^\be y\left(y\eta + w^\g x\right)z + w^\al y^3 + w^{3\be - \al}\left(y\eta + w^\g x\right)^3,
\end{equation*}
where $\eta = \eta(w,x,y)$, $2\al < 3\be$ and $\al$ is not divisible by $3$.
\end{enumerate}
\end{enumerate}
\end{theorem}

\begin{remarks}\label{rem:abjung}
(1) Given that $\inv(a) = \inv(\mathrm{nc}3)$, $a \in X_{j_0}$, and that $X_{j_0}$ is 
generically $\mathrm{nc}3$ on the component of $(\inv = \inv(\mathrm{nc}3))$ containing
$a$, then we can choose coordinates satisfying the hypotheses of the theorem.
\smallskip

(2) The condition (b) in item (2) or (3) of Theorem \ref{thm:char} is reminiscent
of the Abhyankar--Jung Theorem (cf. \cite{Ab}). Note that the implication (a) $\Rightarrow$
(b) in (2) or (3) is not true if we weaken the hypothesis (iii) by assuming only that $f$
has order $3$ at $a$. For example, $f(w,x,y,z)=(z+x)\left(z^2 + (w+y)y^2\right)$ does not
have three analytic branches after substituting $w^2$ for $w$.
\end{remarks}

\begin{theorem}\label{thm:cleaned}
With the hypotheses of Theorem \ref{thm:char}, assume in addition that $\inv(\mathrm{nc}3)$
is the maximum value of $\inv$ on $X_{j_0}$. Then there is a morphism $\s: Z' \to Z_{j_0}$
given by a finite sequence of admissible blowings-up of $X_{j_0} \subset Z_{j_0}$ with centres
in the exceptional divisor, such that $\s^{-1}(a)$ intersects the strict transform
of $(\inv = \inv(\mathrm{nc}3))$ in a single point $a'$, and $X'$ is defined at $a'$ by one of the following equations in 
$\cS$, according to the corresponding case of Theorem \ref{thm:char}:
\begin{enumerate}
\item $xyz = 0$  \quad \emph{nc3};
\item $x(z^2 + wy^2) = 0$ \quad \emph{prod};
\item $z^3 + wy^3 + w^2x^3 -3wxyz = 0$  \quad \emph{cp3}.
\end{enumerate}
\end{theorem}

Theorem \ref{thm:cleaned} follows from Theorem \ref{thm:char} by applying the
cleaning lemma \cite[Sect.\,2]{BMminI} to the normal forms in the latter. The cleaning
lemma is applied in exactly the same way as in \cite[\S\S4.2,\,4.3]{BMminI}, so we
refer the reader to \cite{BMminI} for details. The cleaning lemma will be used in the
same way again in several other steps of the proofs of Theorems \ref{thm:mincompl4var}
and \ref{thm:minstrongcompl4var} in Section 4 below.

Basic properties of cyclic singularities are presented in Section 2 following. Theorem
\ref{thm:char} will be proved in Section 3.

\section{Cyclic singularities}\label{sec:cp}

Let $X$ denote a hypersurface of dimension $n-1$; i.e., $X$ is defined
locally by an equation in $n$ variables. Then $n$-fold normal crossings
singularities nc$n$ of $X$ are isolated, and the locus of $(n-1)$-fold normal
crossings points is a smooth curve. A \emph{cyclic singularity} or \emph{cyclic point}
of \emph{order}
$n-1$, denoted cp($n-1$), is an irreducible singularity that occurs as a limit of nc($n-1$)
points of a hypersurface in $n$ variables, and which cannot be eliminated without
blowing up nc($n-1$) points. 

The cyclic singularity cp$k$ of order $k$ is related to the action of the cyclic group
$\IZ_k$ of order $k$ on $\IC^k$ by permutation of coordinates. In \S\ref{subsec:cp3}
following, we define cp$3$, which is needed for this article,
but it will be clear how to generalize the
construction to arbitrary $k$. The reader should check that cp$2$ $=$ pp.

\subsection{Cyclic points of order $3$} \label{subsec:cp3}
Consider the action of $\IZ_3$ on $\IC^3$ by permutation of coordinates. $\IZ^3$ is
generated by the cyclic permutation $\rho = (1,2,3)$; in terms of the coordinates
$(X,Y,Z)$ of $\IC^3$, $\rho(X,Y,Z) = (Z,X,Y)$. 

The matrices in any finite abelian subgroup of the general linear group $\text{GL}(k,\IC)$
can be diagonalized simultaneously \cite[Prop.\,2.7.2]{Stu}. A diagonalization of the action of $\IZ_3$
on $\IC^3$ is given by
\begin{align}\label{eq:diag}
y_0 &= \frac{1}{3}(X + Y + Z)\nonumber\\
y_1 &= \frac{1}{3}(X + \ep Y + \ep^2 Z)\\
y_2 &= \frac{1}{3}(X + \ep^2 Y + \ep Z)\nonumber
\end{align}
(the discrete Fourier transform); in other words,
$$
y_i \circ \rho = \ep^i y_i, \quad i = 0,1,2,
$$
where $\ep$ denotes the cube root of unity $\ep = e^{2\pi i/3}$.

It is easy to write a set of generators of the algebra of invariant polynomials
for the diagonalized action. Following is a set of basic invariants for the action of $\IZ_3$ 
above:
\begin{equation}\label{eq:invariants}
y_0,\, y_1y_2,\, y_1^3,\, y_2^3.
\end{equation}

Consider the inverse linear transformation of (\ref{eq:diag}):
\begin{align}\label{eq:inverse}
X &= y_0 + y_1 + y_2\nonumber\\
Y &= y_0 + \ep^2 y_1 + \ep y_2\\
Z &= y_0 + \ep y_1 + \ep^2 y_2. \nonumber
\end{align}

Let $\Phi (y_0,y_1,y_2)$ denote the polynomial $XYZ$ obtained from (\ref{eq:inverse}).
Then $\Phi (y_0,y_1,y_2)$ is invariant with respect to the action of $\IZ_3$, so
it is a polynomial in the basic invariants (\ref{eq:invariants}). Therefore
$\Phi (z, w^{1/3}y, w^{2/3}x)$ is a polynomial in $(x,y,z,w)$. 

\begin{definition}\label{def:cp3}
The \emph{cyclic singularity} cp$3$ of \emph{order} $3$ is defined by
$$
\Phi (z, w^{1/3}y, w^{2/3}x) = 0;
$$
in other words, from (\ref{eq:inverse}), by
\begin{equation}\label{eq:cp3}
z^3 + wy^3 + w^2x^3 -3wxyz = 0.
\end{equation}
\end{definition}
\smallskip

\subsection{Singularities in a neighbourhood of a cyclic point}\label{subsec:cpnbhd}
Consider the hypersurface $X \subset \IA^4$ defined by (\ref{eq:cp3}).
Then $X$ has a cp$3$ singularity at the origin. When $w \neq 0$, $X$ is
nc$3$ along the $w$-axis, and has only nc$2$ singularities outside the $w$-axis.

On the other hand, $\Sing X \cap (w=0)$ is the nonzero $x$-axis. We will
show that $X$ has degenerate pinch points along the nonzero $x$-axis.

For $x \neq 0$, write
$$
\eta = \frac{y}{2x}, \quad \zeta= \frac{z}{x},
$$
so that (\ref{eq:cp3}) can be rewritten as
$$
w^2 + 2(4\eta^3 - 3 \eta\zeta)w + \zeta^3 = 0,
$$
or, after completing the square, as
$$
\left(w + 4\eta^3 - 3 \eta\zeta\right)^2 + \zeta^3 - \left(4\eta^3 - 3 \eta\zeta\right)^2 = 0.
$$
Now,
$$
\zeta^3 - \left(4\eta^3 - 3 \eta\zeta\right)^2 
= \left((\zeta - 3\eta^2) - \eta^2\right)^2 \left((\zeta - 3\eta^2) + 2\eta^2\right).
$$
\smallskip

In other words, if we make a change of coordinates
\begin{align*}
y' &= \frac{1}{2}\left(\frac{y}{x}\right),\\
z' &= \frac{z}{x} - \frac{3}{4}\left(\frac{y}{x}\right)^2,\\
w' &= w + \frac{1}{2}\left(\frac{y}{x}\right)^3 - \frac{3}{2}\left(\frac{y}{x}\right)\left(\frac{z}{x}\right)
\end{align*}
when $x \neq 0$, then (\ref{eq:cp3}) can be rewritten (after dropping primes) as
\begin{equation}\label{eq:newcoords}
w^2 + \left(z -  y^2\right)^2\left(z + 2y^2\right) = 0.
\end{equation}
This equation defines a degenerate pinch point when $y = z = w = 0$; i.e., 
(\ref{eq:cp3}) has a degenerate pinch point when $y = z = w = 0$, $x \neq 0$, as claimed.
We also see that a cyclic point cp$3$  (\ref{eq:cp3}) becomes dpp after blowing up the 
nc$3$-axis $(x = y = z = 0)$.

Note that the hypersurface $(w = 0)$ with respect to the coordinates of (\ref{eq:cp3})
becomes $(w + 5y^3 + 3yz = 0)$ in the new coordinates used in 
\eqref{eq:newcoords}.

\section{Limits of triple normal crossings points}\label{sec:proof}

A proof of Theorem \ref{thm:char} will be given in this section. Item (1)
of the theorem is proved in \emph{Resolution except for minimal singularities I} (see \cite[Lemma 3.4]{BMminI}), so we only have to prove (2) and (3). 

\subsection{Normal forms}\label{subsec:proof1}
The following lemma isolates parts of Theorem \ref{thm:char} (2),(3) that are proved 
in this subsection. The proof of the theorem is completed in \S\ref{subsec:proof2}.

\begin{lemma}\label{lem:char}
With the hypotheses of Theorem \ref{thm:char}, assume that $f$ does not split (i.e., does
not have three
local analytic branches) at $a=0$. Then we have the following conclusions. (The statements
following are enumerated as in Theorem \ref{thm:char}.)
\begin{enumerate}
\item[(2)] The following are equivalent:
\smallskip
\begin{enumerate}
\item[(b)] $f(v^2,x,y,z)$ has three analytic branches at $0$;
\item[(c)] after an \'etale coordinate change preserving $(w=0)$, we can write $f(w,x,y,z)$ 
either as
\begin{equation}\label{eq:2norm1}
\left(z+w^\al x\right)\left(z^2 + w^{2\al +1}\left(x\xi + w^\be y\right)^2\right),
\end{equation}
where $\xi = \xi(w,x,y)$, or as
\begin{equation}\label{eq:2norm2}
\left(z+w^\al\left(y\eta + w^\be x\right)\right)\left(z^2 + w^{2\al +1}y^2\right),
\end{equation}
where $\eta = \eta(w,x,y)$.
\smallskip
\end{enumerate}
\item[(3)] The following are equivalent:
\smallskip
\begin{enumerate}
\item[(b)] $f(v^3,x,y,z)$ has three analytic branches at $0$;
\item[(c)] after an \'etale coordinate change preserving $(w=0)$, we can write $f(w,x,y,z)$ as
\begin{equation}\label{eq:3norm}
z^3 - 3w^\be y\left(y\eta + w^\g x\right)z + w^\al y^3 + w^{3\be - \al}\left(y\eta + w^\g x\right)^3,
\end{equation}
where $\eta = \eta(w,x,y)$, $2\al < 3\be$ and $\al$ is not divisible by $3$. (In particular,
$f$ is irreducible.)
\end{enumerate}
\end{enumerate}
\end{lemma}

\begin{proof}
In both (2) and (3), it is clear that (c) $\Rightarrow$ (b). So in each case we will
assume (b) and prove (c). We can assume that 
\begin{equation}\label{eq:tzirn}
f(w,x,y,z) = z^3 + B(w,x,y)z + C(w,x,y),
\end{equation}
and that $V(z) = (z=0)$ is a maximal contact hypersurface at $a=0$.
\smallskip

(2) For any root $\al(v,x,y)$ of $f(v^2,x,y,z)=0$, $\al(-v,x,y)$ is also a root.
Therefore $f(v^2,x,y,z)$ has the form
\begin{equation*}
f(v^2,x,y,z) = (z + \al(v,x,y))(z + \al(-v,x,y))(z + 2\be(v^2,x,y)),
\end{equation*}
where
\begin{equation*}
\be(v^2,x,y) = -\frac{1}{2}(\al(v,x,y) + \al(-v,x,y)),
\end{equation*}
by \eqref{eq:tzirn}. 

Now, we can write
$$
\frac{1}{2}(\al(v,x,y) - \al(-v,x,y)) = v\g(v^2,x,y).
$$
Therefore,
\begin{equation*}
\al(v,x,y) = -\be(v^2,x,y) + v\g(v^2,x,y),
\end{equation*}
and
\begin{equation*}
f(v^2) = (z + 2\be(v^2))(z - \be(v^2) + v\g(v^2))
                     (z - \be(v^2) - v\g(v^2))
\end{equation*}
(where $f(v^2)$ means $f(v^2,x,y,z)$, etc.). Therefore, after a change of coordinates
$z' = z - \be(w,x,y)$,
we can write $f$ in the form
\begin{equation}\label{eq:2factors}
f(w,x,y,z') = \left(z' + \be'(w,x,y)\right)\left((z')^2 - w\g'(w,x,y)^2\right).
\end{equation}

It follows from \eqref{eq:2factors} that the first coefficient (marked) ideal is 
equivalent to $(\be', 1) + (w(\g')^2, 2)$, and that this marked
ideal is of maximal order after factoring $w$ (because $\inv(a) = (3,0,1,\ldots)$). 
Therefore, either $\be'$ has order $1$ after factoring $w$ and we get the normal form \eqref{eq:2norm1},
or $\g'$ has order $1$ after factoring $w$ and we get \eqref{eq:2norm2}.
\smallskip

(3) Since $f$ does not split but $f(v^3,x,y,z)$ splits at $a$, it follows that
the latter factors as 
\begin{align*}
f(v^3,x,y,z) &= (z + \al(v,x,y))(z + \al(\ep v,x,y))(z + \al(\ep^2 v,x,y))\\
                    &= XYZ, \,\, \text{say},
\end{align*}
where $\ep = e^{2\pi i/3}$. Define $y_0, y_1, y_2$ by the formulas (\ref{eq:diag}). Then
\begin{align*}
y_0 &= z\\
y_1 &= \frac{1}{3}\left(\al(v,x,y) + \ep\al(\ep v,x,y) + \ep^2  \al(\ep^2v,x,y)\right)\\
y_2 &= \frac{1}{3}\left(\al(v,x,y) + \ep^2\al(\ep v,x,y) + \ep  \al(\ep^2v,x,y)\right).
\end{align*}

Clearly, $vy_1$ and $v^2y_2$ are $\IZ_3$-invariant (with respect to the action on
the $v$-variable), so that
\begin{align*}
vy_1 &= \eta_1(v^3,x,y)\\
v^2y_2 &= \eta_2(v^3,x,y),
\end{align*}
and we can write
\begin{align*}
y_1 &= v^{3m+2} \zeta_1(v^3,x,y)\\
y_2 &= v^{3n+1} \zeta_2(v^3,x,y),
\end{align*}
where $m,n \geq 0$.

Now consider 
\begin{align*}
\zeta_1 &= \zeta_1(w,x,y)    &y_1 = w^{m+2/3} \zeta_1\\
\zeta_2 &= \zeta_2(w,x,y)    &y_2 = w^{n+1/3} \zeta_2;
\end{align*}
both $\zeta_1$ and $\zeta_2$ are in the ideal generated by $x, y$. By (\ref{eq:inverse}),
$$
f(w,x,y,z) = z^3 + w^{3n+1}\zeta_2^3 + w^{3m+2}\zeta_1^3 - 3 w^{m+n+1}\zeta_1 \zeta_2 z.
$$

Then the first coefficient ideal is equivalent to
$\left((w^{3m+2}\zeta_1^3, w^{3n+1}\zeta_2^3), 3\right)$ 
(cf.\cite[Ex.\,5.13]{BMminI}). 
Set $\al = \min\{3m+2, 3n+1\}$. Let $\zeta_1'$ denote the corresponding $\zeta_i$,
and $\zeta_2'$ the other.
Since $\inv(a) = (3,0,1,\ldots)$, $\zeta_1'$ has order $1$ at $a$, and we can assume that
$\zeta_1' = y$. Since $\inv(a) = (3,0,1,0,1,\ldots)$, it follows that 
$\zeta_2'|_{(y=0)}$ has order $1$ after dividing by $w$ as much as possible. By a further
coordinate change, we get \eqref{eq:3norm}, where $\be = m+n+1$. Note that
\eqref{eq:3norm} would split if $\al$ were divisible by $3$.
\end{proof}

\subsection{Splitting lemmas}\label{subsec:proof2}

In this subsection, we complete the proof of Theorem \ref{thm:char}. We use the
notation of the latter. We can also assume that
\begin{equation}\label{eq:tzirn1}
f(w,x,y,z) = z^3 - 3B(w,x,y)z + C(w,x,y),
\end{equation}
and that $V(z) = (z=0)$ is a maximal contact hypersurface at $a=0$. Set
\begin{equation*}
\D := C^2 - 4B^3;
\end{equation*}
i.e., $-27\D$ is the discriminant of $f$ as a polynomial in $z$. Then the first
coefficient (marked) ideal is 
$$
I := \left((B^3, C^2), 6\right) =  \left((C^2, \D), 6\right).
$$
Since $\inv(a) = (3,0,1,\ldots)$, we have $I = w^\g \tI$, where $\tI$ has order
6 at $a=0$.

The coordinate system $(w,x,y,z)$ induces an identification of the completed
local ring $\wcO_{Z,a}$ with the formal power series ring $\uk\llbr w,x,y,z\rrbr$.
Let $\uk\llb w\rrb$ denote the field of fractions of $\uk\llbr w\rrbr$, and let 
$\overline{\uk\llb w\rrb}$ denote the algebraic closure of $\uk\llb w\rrb$. 
Then 
$\overline{\uk\llb w\rrb}$ is the field of formal Puiseux 
series in $w$ over $\uk$; i.e.,
formal Laurent series over $\uk$ in $w^{1/n}$, with finitely many negative exponents, 
where $n$ ranges over the nonnegative integers. Set
\begin{align*}
R &:= \uk\llbr w,x,y\rrbr,\\
S &:= \overline{\uk\llb w\rrb}\llbr x,y\rrbr.
\end{align*}

Then $f$ splits in $S[z]$; say,
$$
f = (z + \al_0)(z + \al_1)(z + \al_2).
$$
Moreover, each $\al_j$ belongs to the ideal $(x,y)$ generated by $x$ and $y$ in $S$,
by the normal crossings hypothesis (ii) in Theorem \ref{thm:char}. 
Define
\begin{equation*}
\eta_i := \frac{1}{3}\sum_{j=0}^2 \ep^{ij}(z+ \al_j), \quad i = 0,1,2,
\end{equation*}
where $\ep =  e^{2\pi i/3}$ (cf. \eqref{eq:diag}). Then $\eta_0 = z$ and
\begin{align}\label{eq:factors}
f &= \prod_{i=0}^2\left(z + \ep^i\eta_1 + \ep^{2i}\eta_2\right)\\
  &= z^3 -3\eta_1\eta_2 z + \eta_1^3 + \eta_2^3 \notag
\end{align}
in $S[z]$ (cf. \eqref{eq:inverse}). In particular,
$$
B = \eta_1\eta_2, \quad C = \eta_1^3 + \eta_2^3, \quad \D = \left(\eta_1^3 - \eta_2^3\right)^2
$$
in $S$. The preceding notation will be used throughout this section.

\begin{lemma}\label{lem:disc}
Under the hypotheses of Theorem \ref{thm:char},
$\D$ factors in a sufficiently small \'etale neighbourhood of $a$ as
$$
\D = \Phi^2 \Psi,
$$
where $\Psi$ is not in the ideal generated by $x,y$.
\end{lemma}

\begin{proof}
By the normal crossings hypothesis (ii) in Theorem \ref{thm:char}, $\D$ is
a square at the generic point of $(x=y=0)$. The assertion follows.
\end{proof}

Given $\theta \in R$, let $\ord\, \theta$ denote the order of $\theta$ with respect to
the maximal ideal $(w,x,y)$, and let $\ord_{(x,y)} \theta$ denote the order with
respect to the ideal $(x,y)$. Thus, $\ord\, \theta > 0$ if and only if
$\theta$ is not a unit in $R$, and $\ord_{(x,y)} \theta > 0$ if and only if
$\theta$ is not a unit in $S$.

We will prove the following three lemmas, \emph{all under the hypotheses of
Theorem \ref{thm:char}}.

\begin{lemma}\label{lem:Dsquare}
Assume that $\D$ is a square in $R$. Then $f(v^3, x,y,z)$ splits at $a=0$.
\end{lemma}

\begin{lemma}\label{lem:ordcoeff1}
Assume that $\ord\, B^3 > \ord\, C^2$. Then $\D$ is a square in $R$.
\end{lemma}

\begin{lemma}\label{lem:ordcoeff2}
Assume that $\ord\, B^3 \leq \ord\, C^2$. Then $f(v^2, x,y,z)$ splits at $a=0$. 
\end{lemma}

Theorem \ref{thm:char} is an immediate consequence of the preceding three lemmas
and Lemma \ref{lem:char}. Proofs of Lemmas \ref{lem:Dsquare}--\ref{lem:ordcoeff2}
follow. The latter is the most delicate. 

\begin{proof}[Proof of Lemma \ref{lem:Dsquare}]
Write $\D = A^2 \in R$; we can take $A = \eta_1^3-\eta_2^3$. Recall 
$I = (B^3, C^2) = (\D, C^2)= w^\g \tI$, as above. 
Then  $\g = \min\{\ord_{(w)} A^2, \ord_{(w)} C^2\}$.
Therefore, $\g$ is even; say $\g = 2\al$.

We have $4B^3 = (C-A)(C+A)$.

We claim that $w^{-\al}C$ and $w^{-\al}A$ are relatively prime in R:
It is easy to check they are relatively prime in $S$ since $A = \eta_1^3 - \eta_2^3$, $C = \eta_1^3 + \eta_2^3$, and the ideal $(\eta_1,\eta_2)=(x,y)$ in $S$.
Since $\tilde I$ has order $6$, either $\ord\, w^{-\g} \D = \ord_{(x,y)} w^{-\g} \D$ or 
$\ord\, w^{-\g} C^2 = \ord_{(x,y)} w^{-\g} C^2$. In either case, we can use Lemma
\ref{lem:ordgeom} following to conclude that $w^{-\al}C$, $w^{-\al}A$ are 
relatively prime in R.

Therefore, $w^{-\de}(C-A) = 2w^{-\de}\eta_2^3$ and $w^{-\de}(C+A) = 2w^{-\de}\eta_1^3$
are relatively prime in $R$, where $\de$ denotes the largest power of $w$ that divides
$C-A$ and $C+A$.  Moreover, their product $4 w^{-2\de} B^3$ is a cube times a power of 
$w$ in $R$. Hence both $\eta_1^3$ and $\eta_2^3$ are cubes (times powers of $w$)
in $R$. By \eqref{eq:factors}, $f(v^3,x,y,z)$ splits in $R$ and the result follows.
\end{proof}

\begin{lemma}\label{lem:ordgeom}
Let $G \in R$. Suppose that $\ord\, G = \ord_{(x,y)} G$. Let $\theta \in R$ be a nonunit
which divides $G$. Then $\theta$ is also a nonunit in $S$.
\end{lemma}

\begin{proof} 
Consider a decomposition of $G$ into irreducible factors in $R$,
$G = \prod \theta_i^{n_i}$ ,
where the $n_i$ are positive integers. For all $i$, 
$\ord\, \theta_i \geq \ord_{(x,y)} \theta_i$.
By the hypothesis, $\sum n_i \ord\, \theta_i = \sum n_i \ord_{(x,y)} \theta_i$. Therefore,
$\ord\, \theta_i = \ord_{(x,y)} \theta_i$, for all $i$.
The result follows. 
\end{proof}

\begin{proof}[Proof of Lemma \ref{lem:ordcoeff1}]
Since $\ord\, C^2 < \ord\, B^3$, $\ord\, C^2 = \g + 6$. Therefore, $\g$ is even (say $\g = 2\al$),
and $\D = w^{2\al}\widetilde{\D}$, where $\ord\, \widetilde{\D} = 6
= \ord_{(x,y)}\widetilde{\D}$.
By Lemma \ref{lem:disc}, $\D$ is a square in $R$.
\end{proof}

\begin{proof}[Proof of Lemma \ref{lem:ordcoeff2}]
Since $\ord\, B^3 \leq \ord\, C^2$, $\ord\, B^3 = \g+6$. Therefore, $\g$ is divisible by $3$; 
say $\g = 3\al$. Recall that $(z=0)$ is a maximal contact hypersurface for $f$, 
$I = \left((B^3, C^2), 6\right) = w^{\g}\tI$ is the associated coefficient ideal, and $\tI = (\tI, 6)$ is the companion ideal. (The latter has maximal order since $\inv(a)
= (3,0,1,\ldots)$.) By the Weierstrass preparation theorem, 
we can assume that
\begin{equation*}
B = w^\al u ( y^2 - Px^2 ), 
\end{equation*}
where $P = P(w,x)$ and $u$ is a unit.

We will prove that $P$ is a power of $w$ times a unit.
First note that $P$ is a unit in $S$ because $(\eta_1,\eta_2) = (x,y)$ in $S$, so the 
initial form of $B = \eta_1\eta_2$ in $S$ is a non-degenerate quadratic form.
In particular, $P$ has a square root in $S$. Since $S$ is a UFD, we can write
\begin{equation*}
\eta_1 = u_1(y+\sqrt{P}x), \quad \eta_2 = u_2(y-\sqrt{P}x),
\end{equation*}
where $u_1, u_2$ are units of $S$.

Since $\tI \subset (y^2  + Px^2, 2)$, $(y=0)$ is a maximal contact hypersurface for $\tI$ in $(z=0)$.
The associated coefficient ideal of $\tI$ is
\begin{equation*}
J := \left(Px^2,2\right)+ \left(\left.\frac{\partial^2 C}{\partial y^2}\right|_{y=0}, 1 \right) + \left(\left.\frac{\partial C}{\partial y}\right|_{y=0},2 \right) + \left(\left.C\right|_{y=0}3 \right).
\end{equation*}
Since $C = u_1^3(y+\sqrt{P}x)+u_2^3(y-\sqrt{P}x)^3$, a direct computation shows that the marked ideal $J$ reduces to $\left(Px^2,2\right)$.

Since $\inv(a) = (3,0,1,0,1,\ldots)$, the marked ideal $J$ has order $1$ after fully factoring
$w$, i.e., $P$ is a unit times a power $w^\be$ of $w$ in $R$.
Therefore we can assume that 
\begin{equation*}
B = w^\al (y^2 - w^\be x^2),
\end{equation*}
and we can write
\begin{equation*}
\eta_1 = u_1(y+w^{\be/2}x), \quad \eta_2 = u_2(y-w^{\be/2}x).
\end{equation*}

We now substitute $w = v^2$. So we can assume that $B = v^{2\al}(y^2 - v^{2\be} x^2)$.
We will prove that $f(v^2,x,y,z)$ splits.  

First we note that it is enough to prove that $\Delta(v^2,x,y)$ is a square in
$\uk\llbracket v, x, y \rrbracket$:
The latter implies that $f(v^6,x,y,z)$ splits, by Lemma \ref{lem:Dsquare},
and therefore that either $f(v^2,x,y,z)$ or $f(v^3,x,y,z)$ splits, since 
$f$ is a polynomial of degree $3$ in $z$.
Recall that $I = (B^3,C^2) = w^{\g}I'$, where $\g$ is divisible by $3$. 
Then $C$ is divisible by $w^{3\de}$, for some $\de$. If $f(v^3,x,y,z)$ splits,
this would contradict (b) $\Rightarrow$ (c) in Lemma \ref{lem:char}(3).

Let $F$ denote the field of fractions of $R$, $M$ the field of fractions of $S$, and $L$ the subfield of $M$ generated over $F$ by $\eta_1$ and $\eta_2$ --- i.e., the splitting 
field of $f = f(v^2,x,y,z)$.
We consider the Galois group $\Gal_F L$ of $L$ over $F$.

The Galois group $\Gal_F L$ is a subgroup of the symmetric group $\mathfrak S_3$,
where we can view the latter as the group of the permutations of
\[\left\{\eta_1,\ep\eta_1,\ep^2\eta_1,\eta_2,\ep\eta_2,\ep^2\eta_2\right\}\]
preserving the expresions $\eta_1\eta_2$ and $\eta_1^3+\eta_2^3$.

We will prove that $\Gal_F L$ has no element of order $2$.
Consider $\sigma \in \Gal_F L$.
Now, $y\pm v^{\beta}x \in F$,
so $\sigma \eta_1 = (\sigma u_1) (y+v^{\beta}x)$.
By Lemma \ref{lem:preserveunits} following, $\sigma u_1$ is a unit of $S$.
Therefore, $\sigma \eta_1$ cannot be either $\eta_2$, $\ep\eta_2$ or $\ep^2\eta_2$, so that
$\sigma$ cannot be of order $2$.

As a consequence, $\eta_1^3$ and $\eta_2^3$ are fixed by $\Gal_F L$; therefore,
$\eta_1^3, \eta_2^3 \in F$.
So $\Delta$ has a square root in $F$, namely, $\eta_1^3 - \eta_2^3$.
Since $R$ is a unique factorization domain and $F$ is its field of fractions, $\Delta$ also has a square root in $R$.
\end{proof}

\begin{lemma}\label{lem:preserveunits}
Let $S^\times$ denote the set of units of $S$.
Let $\sigma$ be an automorphism of the field of fractions of $S$. Then $\sigma( S^\times ) = S^\times$ and $\sigma S = S$.
\end{lemma}

\begin{proof}
Let $M$ denote the field of fractions of $S$.
As subsets of $M$, the sets $S$ and $S^\times$ admit the following caracterizations, due to the fact that $S$ is a formal power series ring over an algebraically closed field:
$S^\times  = \left\{ f\in M:  \text{ for all } n\in\IN, \text{ there exists } y\in M \text{ such that }
f = y^n \right\}$, and
$S = \left\{ f\in M:  f\in S^\times \text{ or } 1+f\in S^\times \right\}$.
The lemma follows.
\end{proof}

\section{Minimal singularities in four variables}\label{sec:dim4}
In the section, we will prove Theorems \ref{thm:mincompl4var} and 
\ref{thm:minstrongcompl4var} using Theorems \ref{thm:char} and \ref{thm:cleaned}.

 \begin{proof}[Proof of Theorem \ref{thm:mincompl4var}]
 We can reduce to the case that $X$ is an embedded hypersurface
(see Remark \ref{rem:hypersurf}). We then construct the blowing-up sequence in
several steps. (It is possible also to find local normal forms for the minimal
singularities of the total transform; see Remark \ref{rem:mincompl4var}
following the proof.) Let $\cQ_0 := \{\mathrm{nc}4\}$.
\medskip

\noindent
(I) Following the desingularization algorithm, we can blow up with closed admissible
centres lying over the complement
of $\cQ_0$ until the maximum value of the invariant over the complement of $\cQ_0$
is $\inv(\mathrm{nc}3) := (3,0,1,0,1,0,\infty)$. The locus $(\inv = \inv(\mathrm{nc}3))$
is a smooth curve, and its components where $X$ is not generically nc$3$ are closed.
So we can blow up these components.
(For simplicity of notation, we use $X$ to mean also its strict transform in any year of
the blowing-up history, $\cQ_0$ to mean the inverse image of $\cQ_0$, etc.) 

Now, by Theorem \ref{thm:cleaned}; i.e., by the cleaning lemma \cite[Sect.\,2]{BMminI}
applied to the normal forms in Theorem \ref{thm:char}, there is a further sequence
of admissible blowings-up after which
every point of the strict transform of the locus $(\inv = \inv(\mathrm{nc}3))$ is of one of the
following three kinds (where an asterisk in the table means that the exceptional divisor may 
or may not be present at the indicated point).
\medskip

\renewcommand{\arraystretch}{1.2}
\begin{tabular}{r l | l}
             &  singularity  & \,exceptional divisor\\\hline
      nc$3$ & $xyz=0$ & \,$(w=0)^*$\\
      prod & $x(z^2 + wy^2) = 0$ & \,$(w=0)$\\
      cp$3$ & $z^3 + wy^3 + w^2x^3 -3wxyz = 0$\, & \,$(w=0)$     
\end{tabular}
\medskip

The following table lists the singlarities which occur in small neighbourhoods
of $\cQ_0$ and the points in the preceding table. (The equations in the following
table are in suitable \'etale coordinates at the indicated singular points, not 
necessarily with respect to the coodinates in the preceding table.)
\medskip 

\renewcommand{\arraystretch}{1.2}
\begin{tabular}{r | r l | l}
        & & singularity & \,exceptional divisor\\\hline
    nc$4$\, & \,nc$3$ & & \\
                 & \,nc$2$ & & \\
    nc$3$\, & \,nc$2$ & $xy=0$ & \,$(w=0)^*$\\
    prod\,   & \,nc$3$ & & \\
                & \,nc$2$ & & \\
                & \,nc$2$ & $x(z^2 + w)=0$\,   & \,$(w=0)$\\
                & \,pp       & $z^2 + wy^2=0$\, & \,$(w=0)$\\
    cp$3$\, & \,nc$3$ & & \\
                 & \,nc$2$ & & \\
                 & \,dpp     & & \,$(w=0)$          
\end{tabular}
\medskip

\noindent
The absence of an exceptional divisor in any row of the table means, of course, that
the indicated singularity is outside the exceptional divisor shown in the preceding table.
In the case of nc$3$ in this table, the exceptional divisor (if present) is transverse
to smooth points in a small neighbourhood. In the case of prod, there are neighbouring
smooth points $z^2 + w = 0$ with tangent exceptional divisor $(w=0)$. In the last row
of the table, the exceptional divisor $(w=0)$ is with respect to the coordinates for
the equation of cp$3$ in the first table above. In this case, the degenerate pinch points
dpp occur along the nonzero $x$-axis, and the exceptional divisor has tangential contact
(order $3$) at smooth points.
\medskip

\noindent
(II) Let us say we are now in year $j_1$. Let $\cQ_{j_1} := \{\mathrm{nc}4,\, 
\mathrm{nc}3,\, \mathrm{prod},\,\mathrm{cp}3\} = \text{closure of }\{\mathrm{nc}3\}$. The
points nc$4$, prod and cp$3$ are isolated.

Recall that, according to 
%Remark \ref{rem:limdpp}
\cite[Rmk.\,4.4]{BMminI}, we cannot, in general, reduce
limits of degenerate pinch points to dpp by cleaning. In this step we will show, however,
that, after additional blowings-up of points, 
limiting points of those components of the dpp locus
that are adherent to $\{\mathrm{cp}3\}$ can be cleaned up to give only dpp. At the
same time, we will clean up the components adherent to $\{$prod$\}$ of the locus
of nc$2$ points $x(z^2 + w) = 0$ with tangent exceptional divisor $(w=0)$.      
 
So we blow up $\{\mathrm{prod},\,\mathrm{cp}3\}$. First consider the effect on a prod 
singularity $x(z^2 + wy^2) = 0$. After blowing up, we have a prod singularity with
the same equation, at the origin of the chart with coordinates $(w, x/w, y/w, z/w)$
(the ``$w$-chart''; for economy of notation we use $(w,x,y,z)$ again to denote the
new coordinates). At nonzero points of the $y$-axis in this chart, we have nc$2$
singularities $x(z^2 + w) = 0$ with exceptional divisor $(w=0)$, in suitable \'etale local
coordinates. In the $y$-chart, with coordinates $(w/y, x/y, y, z/y)$ (which we again
denote simply $(w,x,y,z)$), these points occur along the nonzero $w$-axis. 
At the origin of this chart, we have a singularity
$x(z^2 + wy)=0$ with exceptional divisor $(y=0)$. Let us blow up this point. Then
in the new $w$-chart, we get $x(z^2 + y)=0$ with exceptional divisor $(w=0) + (y=0)$. 

Secondly, consider the effect of blowing up a cp$3$ point $z^3 + wy^3 + w^2x^3
- 3wxyz = 0$, where $(w=0)$ is the exceptional divisor. In the $w$-chart, with
coordinates $(w, x/w, y/w, z/w)$, we get a singularity of the same kind at the
origin, with dpp singularities along the nonzero $x$-axis. (According to 
\S\ref{subsec:cpnbhd}, these dpp singularities can be written 
$w^2 + (z-y^2)^2(z+2y^2)=0$, with exceptional divisor $(w+5y^3+3yz = 0)$, in 
suitable local coordinates.) These dpp points occur
along the nonzero $w$-axis of the $x$-chart, with coordinates $(w/x, x, y/x, z/x)$.
At the origin of this chart we have
$$
z^3 + wxy^3 + w^2x^2 - 3wxyz = 0,
$$ 
with exceptional divisor $(x=0)$. Let us blow up this point twice. Then we get
\begin{equation}\label{eq:limdpp}
z^3 + w^3xy^3 + x^2 - 3wxyz = 0,
\end{equation}
with exceptional divisor $(x=0) + (w=0)$ (and the dpp singularities still along
the nonzero $w$-axis). 
Completing the square with respect to $x$, we can rewrite \eqref{eq:limdpp} as
$$
\left(x + \frac{w^3y^3-3wyz}{2}\right)^2 + z^3 - \frac{(w^3y^3-3wyz)^2}{4} = 0,
$$
which is the same thing as
\begin{multline*}
\left(x + \frac{w^3y^3-3wyz}{2}\right)^2 \\
+ \left(z - 3\left(\frac{wy}{2}\right)^2 - \left(\frac{wy}{2}\right)^2\right)^2
\left(z - 3\left(\frac{wy}{2}\right)^2 + 2\left(\frac{wy}{2}\right)^2\right) = 0.
\end{multline*}
Therefore, after a change of variables, we have
\begin{equation}\label{eq:limdpp1}
x^2 + (z - w^2y^2)^2(z + 2 w^2y^2) = 0,
\end{equation}
with exceptional divisor $(x + 5w^3y^3 + 3wyz = 0) + (w=0)$ (compare \S\ref{subsec:cpnbhd}).

We can apply the cleaning lemma to \eqref{eq:limdpp1}: We first blow up $(x=z=w=0)$
twice to get
$$
x^2 + w^2(z-y^2)^2(z+2y^2) = 0,
$$
with exceptional divisor $(x + 5wy^3 + 3wyz = 0) + (w=0)$. Then we blow up
$(x = w =0)$ to get a dpp
$$
x^2 + (z-y^2)^2(z+2y^2) = 0,
$$
with exceptional divisor $(x + 5y^3 + 3yz = 0) + (w=0)$. (In particular, no new 
singularity types occur as limits of dpp points in a neighbourhood of 
$\{\mathrm{cp}3\}$.) The centres of the blowings-up
involved in the cleaning are isolated from $\{\mathrm{nc}4,\,\mathrm{nc}3,\,
\mathrm{prod},\,\allowbreak \mathrm{cp}3\} = \text{closure of }\{\mathrm{nc}3\}$.
\medskip

\noindent
(III) Let us say we are now in year $j_2 \geq j_1$. Let $\cQ_{j_2}$ denote the union
of $\{\mathrm{nc}4,\,\mathrm{nc}3,\,\mathrm{prod},\,\mathrm{cp}3\} = \text{closure of }\{\mathrm{nc}3\}$, the adherent components of the dpp locus, and the adherent components
of the locus of nc$2$ points  $x(z^2 + w) = 0$ with tangential exceptional divisor $(w=0)$.
Then $\cQ_{j_2}$ is a closed subset of (the strict transform in year $j_2$ of) $X$.
(Note that in the current year $j_2$, for each cp$3$ singularity $z^3 + wy^3 + w^2x^3
- 3wxyz = 0$, where $(w=0)$ is the exceptional divisor, the adherent component of
the dpp locus comprises the dpp points which lie in the indicated component $(w=0)$
of the 
exceptional divisor. The purpose of step (II) was to guarantee that these dpp points
together with the cp$3$ points form a closed subset. A similar remark applies to
each prod point and the neighbouring nc$2$ points with tangential exceptional divisor.)

In some neighbourhood of $\cQ_{j_2}$, all singular points outside $\cQ_{j_2}$ are nc$2$,
$xy=0$, with exceptional divisor $(w=0)$ if present (in suitable \'etale local coordinates).
Moreover, all previous blowings-up are $\inv_1$-admissible, so we can extend $\inv_1$
in year $j_2$ to an invariant $\inv$ as usual (i.e., considering $j_2$ to be ``year zero'' for
$\inv_{3/2}$). At a neighbourhing nc$2$ point $a$ outside $\cQ_{j_2}$, we have 
$\inv(a) = \inv(\mathrm{nc}2) := (2,0,1,0,\infty)$ if $a$ lies outside the exceptional
divisor. If $a$ is in the exceptional divisor, we have either $\inv(a) = (2,0,1,1,1,0,\infty)$,
or $\inv(a) = (2,1,1,0,1,0,\infty)$. In either case, the exceptional divisor is transverse to
the nc$2$-locus and therefore transverse to any maximal contact hypersurface.

We can blow up with closed admissible centres outside $\cQ_{j_2}$ until the maximal
value of $\inv$ is $\leq (2,1,1,0,1,0,\infty)$. Then any component of the locus
$(\inv = (2,1,1,0,1,0,\infty))$ which is not generically nc$2$ with transverse exceptional
divisor is separated from $\cQ_{j_2}$, so we can also blow up these components.
It is now not difficult to see that, at any singularity outside $\cQ_{j_2}$ with $\inv =
(2,1,1,0,1,0,\infty)$, we can choose local \'etale coordinates $(v,w,x,y)$ 
in which $X$ and the ``old'' exceptional divisor (counted by $s_1 = 1$ in 
$\inv$) are given (respectively) either by equations of the form
\begin{align*}
z^2 + w^\al y^2 &= 0,\\
\zeta + w^\be (\eta + w^\g v) &= 0,
\end{align*}
where $\al \leq 2\be$, $\zeta$ belongs to the ideal generated
by $z$, $\eta$ belongs to the ideal generated by $y$, and $(w=0)$ is the
``new'' exceptional divisor, or by equations of the form
\begin{align*}
z^2 + w^\al (\nu + w^\g y)^2&= 0,\\
\zeta + w^\be v &= 0,
\end{align*}
where $\al \geq 2\be$, $\zeta$ is in the ideal generated
by $z$, $\nu$ is in the ideal generated by $v$, and $(w=0)$ is the
new exceptional divisor.
(Compare with 
%\S\ref{subsec:inv} \cite[\S1.2]{BMminI} and Lemma \ref{lem:pp}
 \cite[\S1.2 and Lemma 4.2]{BMminI}.)
In both cases, by the cleaning lemma, we can blow up to 
get either nc$2$, $z^2 + y^2 = 0$, or pp, $z^2 + wy^2 = 0$, where $(v=0)$ and
perhaps $(w=0)$ (for example, in the pp case) give the support of the exceptional divisor.

We now repeat essentially 
the same operations using $(2,0,1,1,1,0,\infty)$ for the maximum
value of $\inv$
outside $\cQ_{j_2}$ and the strict transform of $(\inv = (2,1,1,0,1,0,\infty))$ above, then 
again using $\inv(\mathrm{nc}2)$ for the maximum value outside $\cQ_{j_2}$ and the 
strict transforms of the previous two $\inv$-loci. The points outside $\cQ_{j_2}$ that have
been cleaned up are all nc$2$ or pp, with exceptional divisor as indicated above.
\medskip

\noindent
(IV) We can now use the desingularization algorithm to resolve any singularities
remaining outside the locus of points with singularities in $\cS$, by admissible blowings-up.
This completes the proof.
\end{proof}

\begin{remark} \label{rem:mincompl4var}
The proof above provides normal forms for the total transform at every singular
point of the strict transform. In order to get appropriate normal forms for the total
transform at smooth points of the strict transform, 
we should continue as in
%\ref{rem:high3total}
\cite[Rmk.\,4.3]{BMminI}. The normal forms will include the possibility of tangential contact
of a component of the exceptional divisor, of order $2$ (for, example, in a neighbourhood
of a pinch point) or order $3$ (for example, in a neighbourhood of a degenerate
pinch point).
\end{remark}

\begin{proof}[Proof of Theorem \ref{thm:minstrongcompl4var}]
We begin as in Theorem \ref{thm:mincompl4var}, repeat steps (I) and (II)
of the latter, and then continue as follows.
\medskip

\noindent
(III) Let us say we are now in year $j_2 \geq j_1$. Set $\cQ_{j_2}' := \cQ_{j_2} \cup
\{\mathrm{exc}\}$, where $\cQ_{j_2}$ is the subset defined in step (III) of 
Theorem \ref{thm:mincompl4var} and $\{\mathrm{exc}\}$ denotes the set of 
exceptional 
singularities \eqref{eq:excII}. We can blow up with closed admissible centres
outside $\cQ_{j_2}'$ until the maximum value of the $\inv$ over the complement
of $\cQ_{j_2}'$ is $\inv(\mathrm{dpp}) := (2,0,3/2,0,2,0,\infty)$. 

The locus $(\inv = \inv(\mathrm{dpp}))$ is a smooth curve. Each component of
this curve either contains no dpp or is generically dpp (according as
$\Sing X$ has codimension $> 2$ or $=2$ at the generic point). Components
with no dpp are closed and separated from $\cQ_{j_2}'$; we can blow up to
get rid of these components. So any remaining component of 
$(\inv = \inv(\mathrm{dpp}))$ is generically dpp.  

By 
%Remark \ref{rem:limdpp}
\cite[Rmk.\,4.4]{BMminI}, at any point of $(\inv = \inv(\mathrm{dpp}))$, the
(strict transform of) $X$ is defined by an equation of the form
$$
z^2 + w^\al \left(y + w^\be x^2 \right)^2 \left(y - 2w^\be x^2\right) = 0,
$$
in suitable \'etale coordinates $(w,x,y,z)$, 
where $w^\al$, $w^\be$ are monomials
in the exceptional divisor $(w=0)$. (The exceptional divisor has only one
component at the given point, since any component is transverse to
$(\inv = \inv(\mathrm{dpp})$.) 

We can use the cleaning lemma to blow up to reduce to the case 
that $\be = 0$ and $|\al| = 0$ or $1$. (The centres of blowing up involved
are closed in $X$ and disjoint from $\cQ_{j_2}' \cup \{\mathrm{dpp}\}$.)
If $\al = 0$, then we
have a dpp.
If $|\al| = 1$, then we can rewrite
the equation as
\begin{equation}\label{eq:preexc}
z^2 + w y (y + x^2)^2 = 0.
\end{equation}
(We recall that blowing up $(z = y = w = 0)$ results in an exceptional
singularity $z^2 + y(w y + x^2) = 0$.) 
\medskip

\noindent
(IV) Let us say we are now in year $j_3 \geq j_2$. Let $\cQ_{j_3}'$ denote 
$\cQ_{j_2}'$ together with all degenerate pinch points and their limits.
(Limits of dpp outside $\cQ_{j_2}'$ are either dpp or singularities of the 
form \eqref{eq:preexc}.) The blowings-up involved in (III) are 
$\inv_1$-admissible, so we can extend $\inv_1$
in year $j_3$ to an invariant $\inv$ as usual (i.e., considering $j_3$ to be 
year zero for $\inv_{3/2}$; cf. step (III) in the proof of 
Theorem \ref{thm:mincompl4var}). We can blow up with closed admissible centres
outside $\cQ_{j_3}'$ until the maximum value of the $\inv$ over the complement
of $\cQ_{j_3}'$ is $\inv(\mathrm{pp}) := (2,0,3/2,0,1,0,\infty)$. 

We argue as in step (III).
The locus $(\inv = \inv(\mathrm{pp}))$ is a smooth curve. Each component of
this curve either contains no pp or is generically pp. Components
with no pp are closed and separated from $\cQ_{j_3}'$; we can blow up to
get rid of these components. So any remaining component of
$(\inv = \inv(\mathrm{pp}))$ is generically pp. At any point
of $(\inv = \inv(\mathrm{pp}))$, $X$ is defined by an equation of the form
$$
z^2 + w^\al \left(y + w^\be x \right)^2 \left(y - 2w^\be x\right) = 0,
$$
in suitable \'etale coordinates $(w,x,y,z)$,
where $w^\al$, $w^\be$ are monomials
in the exceptional divisor $(w=0)$. We can proceed as in the proof
of 
%Theorem \ref{thm:high3} in Section 4
\cite[Thm.\,1.14]{BMminI}, using the cleaning lemma, to 
reduce to the case that $\al = \be = 0$; i.e., to a pinch point with 
exceptional divisor transverse to the pp locus. 
\medskip

\noindent
(V) Say we are now in year $j_4 \geq j_3$, and set $\cQ_{j_4}' := \cQ_{j_3}' \cup 
\{\mathrm{pp}\}$. We then repeat the argument of Step (III) in Theorem
\ref{thm:mincompl4var}, successively using $\inv = (2,1,1,0,1,0,\infty)$,
$(2,0,1,1,1,0,\infty)$, $(2,0,1,0,\infty)$ as maximum value and then cleaning
the corresponding locus. The points outside $\cQ_{j_4}'$ that have been
cleaned up are all either nc$2$, $z^2 + y^2 = 0$, or pp,
$z^2 + uy^2 = 0$, with appropriate exceptional divisor.
\medskip

\noindent
(VI) Say we are now in year $j_5 \geq j_4$, and define $\cQ_{j_5}'$ by adjoining
the latter points to $\cQ_{j_4}'$. Then $\cQ_{j_5}'$ comprises singularities
in $\cS'$, together perhaps with singularities of type \eqref{eq:preexc}.
Consider a singularity \eqref{eq:preexc}. 
At a nearby point $a$ where $z=y=w=0$,
$y+x^2 \neq 0$, we can choose \'etale coordinates $(w,x,y,z)$ in which we
have $z^2 + wy = 0$ (i.e., $a$ is a quadratic cone singularity) with 
exceptional divisor $(w=0)$. So $\inv(a) = (2,0,1/2,1,1,0,\infty)$.

We can blow up with closed admissible centres separated from $\cQ_{j_5}'$
until the maximum value of $\inv$ in the complement of $\cQ_{j_5}'$ 
is $(2,0,1/2,1,\allowbreak 1,0,\infty)$. 
Then the locus $(\inv = (2,0,1/2,1,1,0,\infty))$ 
extends to a point \eqref{eq:preexc} as above, as $z = y = w = 0$. So this
locus, together with limiting points of the form \eqref{eq:preexc} is a
closed set that provides an admissible centre of blowing up. We blow up
this locus. The effect is to convert singularities \eqref{eq:preexc} to 
exceptional singularities exc. 
\medskip

\noindent
(VII) We can now use the desingularization algorithm to resolve any
singularities remaining outside the locus of points with singularities
in $\cS'$, by admissible blowings-up. This completes the proof.
(Normal forms for the total transform can again be handled as in Remark
\ref{rem:mincompl4var}.)  
\end{proof}

\bibliographystyle{amsplain}

\end{document}